\documentclass{amsart}
\usepackage{amsthm}
\usepackage{amsfonts}
\usepackage{amssymb,latexsym}
\usepackage{amsbsy}
\usepackage{amsxtra}
\usepackage{hyperref}
\usepackage{amsmath}
\usepackage{mathabx}
\usepackage{mathrsfs}
\usepackage{graphicx}

\theoremstyle{plain}
\newtheorem{theorem}{Theorem}

\newtheorem{proposition}{Proposition}
\newtheorem{corollary}{Corollary}

\begin{document}
	
	\title[Improved Spectral Cluster Bounds for Orthonormal Systems
	]{Improved Spectral Cluster Bounds for Orthonormal Systems 
	}

	\begin{abstract}
		We improve Frank-Sabin's work \cite{FS} concerning the spectral cluster bounds for orthonormal systems at $p=\infty$,  on the flat torus and spaces of nonpositive sectional curvature, by shrinking the spectral band from $[\lambda^{2}, (\lambda+1)^{2})$ to $[\lambda^{2}, (\lambda+\epsilon(\lambda))^{2})$, where $\epsilon(\lambda)$ is a function of $\lambda$ that goes to $0$ as $\lambda$ goes to $\infty$. In achieving this, we invoke the method developed by Bourgain-Shao-Sogge-Yao \cite{BSSY}.
	\end{abstract}



\author{Tianyi Ren, An Zhang}

\address{
	School of Mathematical Sciences\\
	Beihang University\\
	9 NanSan Street, ShaHe Higher Education Park
	Beijing, 102206, China
}
\email{rentianyi@buaa.edu.cn, anzhang@buaa.edu.cn}

\date{\today}

\thanks{
This research was supported in part by NSFC grants No.11801536, and the Fundamental Research Funds for the Central Universities Grant No.YWF-22-T-204. 
}
	
	\maketitle
	\section{Introduction}
	
	Let $(M, g)$ be a smooth compact Riemannian manifold of dimension $n\geqslant 2$ without boundary, and let $\Delta_{g}$ be the Laplace-Beltrami operator on $M$. It is well-known that $-\Delta_{g}$ is a nonnegative self-adjoint operator on $L^{2}(M)$ with discrete eigenvalues, which we label as \[0=\lambda_{0}^{2}\leqslant\lambda_{1}^{2}\leqslant\lambda_{2}^{2}\leqslant\cdots,\] 
	counted with multiplicities. For each $j\geqslant 0$, let $e_{j}(x)$ be the $L^{2}$-normalized eigenfunction corresponding to the eigenvalue $\lambda_{j}$, which is a $C^{\infty}$ function on $M$: $-\Delta_{g}e_{j}=\lambda_{j}^{2}e_{j}$. Denote as $E_{j}$ the spectral projection operator onto the eigenspace spanned by all $e_{j}$. Given any $\lambda\geqslant 1$, we define the spectral projection operator associated with $-\Delta_{g}$ onto the band $[\lambda^{2}, (\lambda+1)^{2})$ as: \begin{equation*}
		\chi_{\lambda}f\ :=\sum_{\lambda_{j}\in[\lambda, \lambda+1)}E_{j}f.
	\end{equation*}
	
	In the celebrated work \cite{Sogge}, Sogge proved the following sharp estimate for $\chi_{\lambda}$, which can be viewed as the compact manifold version of the Stein-Tomas inequality in $\mathbb{R}^{n}$:
	\begin{equation} \label{Sogge}
		\|\chi_{\lambda}f\|_{L^{p}(M)}\leqslant C\lambda^{\sigma(p)}\|f\|_{L^{2}(M)},
	\end{equation}
   where $2\leqslant p\leqslant\infty$, and
   \begin{equation} \label{sigma(p)}
   	\begin{aligned}
   		\sigma(p)=\left\{\begin{aligned}
   			&\frac{n-1}{2}(\frac{1}{2}-\frac{1}{p})\quad\quad\mathrm{if}\ 2\leqslant p\leqslant\frac{2(n+1)}{n-1}\,,\\
   			&n(\frac{1}{2}-\frac{1}{p})-\frac{1}{2}\quad\quad\mathrm{if}\ \frac{2(n+1)}{n-1}\leqslant p\leqslant\infty\,.
   		\end{aligned}\right.
   	\end{aligned}
   \end{equation} 
   Letting $V_{\lambda}\ :=\chi_{\lambda}(L^{2}(M))$, which is commonly called a spectral cluster, an important consequence of the above estimate is that for any $g\in V_{\lambda}$, $L^{2}$-normalized for the sake of convenience, one has the $L^p$ bound
   \begin{equation} \label{Sogge1}
   	\|g\|_{L^{p}(M)}\leqslant C\lambda^{\sigma(p)}.
   \end{equation}
    
   Since Sogge's work, the study of $\chi_{\lambda}$ has attracted extensive interests. It is closely related to several other topics in harmonic analysis, such as the restriction and resolvent estimates. In Bourgain-Shao-Sogge-Yao \cite{BSSY}, the authors researched the possibility of shrinking the interval of the spectral projection operator from $[\lambda^{2}, (\lambda+1)^{2})$ to $[\lambda^{2}, (\lambda+\epsilon(\lambda))^{2})$, where $\epsilon(\lambda)$ is a function of $\lambda$ that goes to $0$ as $\lambda$ goes to $\infty$ and proving refined estimates for different types of compact manifold. They gave affirmative answers in two cases: the flat torus and spaces of nonpositive sectional curvature. They then used their finding to widen the region for the parameter $z$ in the resolvent estimates. We refer the reader to \cite{SoggeHangzhou,MR3849647,BHS,BS,MR4445914,MR4324292} and the references therein for more background information and an abundance of progresses in these directions.
	
	Recently, Frank and Sabin  \cite{FS} went beyond estimates for single functions, applying $\chi_{\lambda}$ to an \textit{orthonormal system}. More specifically, let $W$ be a subspace of $V_{\lambda}$ and pick an orthonormal basis $\{g_{i}(x)\}_{i\in I}$ of $W$. Then the square sum
	\[\rho^{W}\ :=\sum_{i\in I}|g_{i}|^{2}\,,\]
	is independent of the choice of basis. \cite{FS} obtained the following estimate of $\rho^{W}$: \begin{equation} \label{FS}
		\|\rho^{V}\|_{L^{p/2}(M)}\leqslant C\lambda^{\sigma(p)}(\mathrm{dim}W)^{1/\alpha(p)},
	\end{equation}
	where $\sigma(p)$ is as in \eqref{sigma(p)}, and
	\begin{equation} \label{alpha(p)}
		\alpha(p)=\left\{\begin{aligned}
			&\frac{2p}{p+2}\quad\quad\mathrm{if}\ 2\leqslant p\leqslant\frac{2(n+1)}{n-1}\,,\\
			&\frac{p(n-1)}{2n}\quad\quad\mathrm{if}\ \frac{2(n+1)}{n-1}\leqslant p\leqslant\infty\,.
		\end{aligned}\right.
	\end{equation}

	Frank-Sabin's estimate is sharp from several perspectives. For instance, with the $\sigma(p)$ in \eqref{FS}, the $\alpha(p)$ is optimal in that it cannot be increased. It is particularly worth pointing out that their estimate is sharp in a strong sense specified in their paper when the manifold is a two dimensional standard round sphere $\mathbb{S}^{2}$. 
	
	When $\mathrm{dim}W=1$, their estimate recovers Sogge's classical result \eqref{Sogge1}; and when $W=V_{\lambda}$, it leads to the classical Weyl formula for $-\Delta_{g}$ due to Avakumovi\'c \cite{Avakumovic}, Levitan \cite{Levitan} and H\"{o}rmander \cite{Hormander}, that counts the number of eigenvalues of $-\Delta_{g}$ that are $\leqslant\lambda^{2}$. Therefore, the above estimate \eqref{FS} provides a bound that interpolates in an optimal way between the two classical extreme cases. Note especially that in \eqref{FS}, $\alpha(p)>1$ for any $p>2$. Therefore, \eqref{FS} is better than that obtained from Sogge's single function result (the case $\mathrm{dim}W=1$) and the triangle inequality. The reason that accounts for the lift from $1$ to $\alpha(p)$ in the exponent, which is ignored by the triangle inequality derivation, is the \textit{orthogonality} of the eigenfunctions. This is the crucial point of Frank-Sabin's estimate. 
	
	The question of generalizing inequalities involving a single function to inequalities involving an orthonormal system finds its motivation in the study of many-body problems in quantum mechanics, where independent fermionic particles are represented by orthonormal functions in some $L^{2}$-space. The first such work \cite{LT}, due to Lieb and Thirring and now famously known as the Lieb-Thirring inequality, extends the classical Gagliardo-Nirenberg-Sobolev inequality to orthonormal systems, and since then, generalizations of restriction estimates, Strichartz inequalities, etc., have also been studied. See, among others, \cite{FLLS,FS2,BLN,Nakamura,mondal2022orthonormal} for a number of recent achievements of this kind.
	
	In our paper, we will, as in \cite{BSSY}, first improve Frank-Sabin's estimate for $p=\infty$ in the special case of the flat torus $\mathbb{T}^{n}$, by shrinking the spectral band from $[\lambda^{2}, (\lambda+1)^{2})$ to $[\lambda^{2}, (\lambda+\epsilon(\lambda))^{2})$. Let $\Delta_{\mathbb{T}^{n}}$ represent the Laplace Beltrami operator on $\mathbb{T}^{n}$. Let
	$\epsilon(\lambda)=\lambda^{-\frac{n-1}{n+1}}$.
	Denote the spectral projection operator for $-\Delta_{\mathbb{T}^{n}}$ onto the shrunken interval $[\lambda^{2}, (\lambda+\epsilon(\lambda))^{2})$ by $\chi_{\lambda}^{\epsilon}$:
	\[\chi_{\lambda}^{\epsilon}f\ :=\sum_{\lambda_{j}\in[\lambda, \lambda+\epsilon(\lambda))}E_{j}f.\]
	Our main theorem is the following:
	\begin{theorem} \label{Thm1}
		Let $U_{\lambda}^{\epsilon}=\chi_{\lambda}^{\epsilon}(L^{2}(\mathbb{T}^{n}))$ and let $R$ be a subspace of $U$. Pick any orthonormal basis $\{g_{j}(x)\}_{J\in J}$ of $R$ and define $\rho^{R}\ :=\sum_{j\in J}|g_{j}|^{2}$, which is independent of the choice of basis. Then we have
		\begin{equation*}
			\|\rho^{R}\|_{L^{\infty}(\mathbb{T}^{n})}\leqslant C\lambda^{n-1}\epsilon(\lambda).
		\end{equation*}
	\end{theorem}
	Summing up $\epsilon^{-1}(\lambda)$ of the above estimate, we recover Frank-Sabin's result \eqref{FS}, hence our theorem is indeed an improvement over their estimate.
	
	Interpolating with the estimate at the critical exponent $p=\frac{2(n+1)}{n-1}$, we obtain
	
	\begin{corollary} \label{Cor1}
			 Let $\epsilon(\lambda)$ and $\rho^{R}$ be as in Theorem \ref{Thm1}. Then for each $\frac{2(n+1)}{n-1}\leqslant p\leqslant\infty$ we have
			\begin{equation} \label{Otherp}
				\|\rho^{R}\|_{L^{p/2}(\mathbb{T}^{n})}\leqslant C\lambda^{2\sigma(p)}\epsilon(\lambda)^{\left(1-\frac{2(n+1)}{p(n-1)}\right)}(\mathrm{dim}R)^{1/\alpha(p)},
			\end{equation}
			where $\alpha(p)=\frac{p(n-1)}{2n}$ and $\sigma(p)=n(\frac{1}{2}-\frac{1}{p})-\frac{1}{2}$ are as in \eqref{alpha(p)} and \eqref{sigma(p)} respectively.
	\end{corollary}
	
	By a similar process to how we will treat the flat torus case, we may as well obtain a $\mathrm{ln}\lambda$ improvement for general compact manifold of nonpositive sectional curvature. Let $(M, g)$ be a Riemannian manifold with nonpositive sectional curvature. We still use $\epsilon(\lambda)$ to denote the width of the spectral band, but here, the width can only be some fixed multiple of $\frac{1}{\mathrm{ln}\lambda}$: $\epsilon(\lambda)=C\frac{1}{\mathrm{ln}\lambda}$. Let $\chi_{\lambda}^{\epsilon}$ be the spectral projection operator associated with the Laplacian $-\Delta_{g}$ onto the shrunken interval $[\lambda^{2}, (\lambda+\epsilon(\lambda))^{2})$:
	\[\chi_{\lambda}^{\epsilon}f\ :=\sum_{\lambda_{j}\in[\lambda, \lambda+\epsilon(\lambda))}E_{j}f.\]
	The main theorem for spaces of nonpositive curvature is the following.
	
	\begin{theorem} \label{Thm6}
		Let $U_{\lambda}^{\epsilon}=\chi_{\lambda}^{\epsilon}(L^{2}(M))$ and let $R$ be a subspace of $U$. Pick any orthonormal basis $\{g_{j}(x)\}_{j\in J}$ of $R$ and define $\rho^{R}\ :=\sum_{j\in J}|g_{j}|^{2}$, which is independent of the choice of basis. Then we have
		\begin{equation*}
			\|\rho^{R}\|_{L^{\infty}(M)}\leqslant C\lambda^{n-1}\epsilon(\lambda).
		\end{equation*}
	\end{theorem}

Again, summing up $\epsilon^{-1}(\lambda)$ of the above estimate, we recover Frank-Sabin's result \eqref{FS}. Likewise, an interpolation with the estimate at $p=\frac{2(n+1)}{n-1}$ gives

\begin{corollary} \label{Cor2}
	Let $\epsilon(\lambda)$ and $\rho^{R}$ be as in Theorem \ref{Thm6}. Then for each $\frac{2(n+1)}{n-1}\leqslant p\leqslant\infty$ we have
	\begin{equation*}
		\|\rho^{R}\|_{L^{p/2}(M)}\leqslant C\lambda^{2\sigma(p)}\epsilon(\lambda)^{\left(1-\frac{2(n+1)}{p(n-1)}\right)}(\mathrm{dim}R)^{1/\alpha(p)},
	\end{equation*}
	where $\alpha(p)=\frac{p(n-1)}{2n}$ and $\sigma(p)=n(\frac{1}{2}-\frac{1}{p})-\frac{1}{2}$.
\end{corollary}

For the proof of Theorem \ref{Thm1} and \ref{Thm6}, we apply the scheme established in \cite{FS} and \cite{FS2}, and then invoke the method developed in \cite{BSSY}. In Section 2, we prove Theorem \ref{Thm1} and Corollary \ref{Cor1}, and in Section 3, we prove Theorem \ref{Thm6} and Corollary \ref{Cor2}.

To conclude the introduction, we remark that our result at the critical exponent $p=\frac{2(n+1)}{n-1}$ is essentially established in \cite{FS} and the power of $\epsilon$ in Corollary \ref{Cor1} is worse than the case of single functions in \cite{BSSY,Hickman}, which is however necessary. On the other hand, there have already been a number of breakthroughs \cite{BSSY,BD,Hickman,GM,BS,BHS} in the study of the spectral projection operator on the torus and spaces of nonpositive curvature applied to single functions at the critical exponent, using recent techniques such as the $l^{2}$-decoupling theorem and bilinear arguments. However, we remark that these critical estimates for single functions, even that of \cite{BSSY}, cannot be easily extended to orthonormal systems, since, as we can see, the power of $\epsilon$ on the right side of \eqref{Otherp} cannot be expected to be $1$, but has to be relaxed when the dimension of the spectral cluster increases.
For example in the case of torus, if we choose $R$ to be the whole spectral cluster, $R=U_\lambda^\epsilon$, then the left side will have a lower bound $\epsilon\lambda^{n-1}$ and therefore the power of $\epsilon$ on the right hand side has to go down below $1-\frac{1}{\alpha(p)}$.
If the mentioned new techniques could be combined well with the method for orthonormal systems, then we would be able to improve the estimate for orthonormal systems at the critical exponent significantly, and by interpolation, obtain better results (in the sense of the choice of $\epsilon$ and its power) at other exponents of $p$. This could be a project for future research.
	
	\maketitle
	\section{Proof of Theorem \ref{Thm1}}
	
	We follow the scheme established in \cite{FS} and \cite{FS2}. First, Theorem \ref{Thm1} immediately follows from the stronger result below by taking all coefficients to be $1$.
	
	\begin{theorem} \label{Thm2}
		Let $\chi_{\lambda}^{\epsilon}$ be the spectral projection operator for $-\Delta_{\mathbb{T}^{n}}$ onto the interval $[\lambda^{2}, (\lambda+\epsilon(\lambda))^{2})$, where $\epsilon(\lambda)=\lambda^{\frac{n-1}{n+1}}$. Let $U_{\lambda}^{\epsilon}=\chi_{\lambda}^{\epsilon}(L^{2}(\mathbb{T}^{n}))$. Then, there exists a $C>0$ such that for any orthonormal system $\{g_{j}(x)\}_{j\in J}$ of $U_{\lambda}^{\epsilon}$ and any sequence of complex numbers $\{\zeta_{j}\}_{j\in J}$, we have \begin{equation*}
			\|\sum_{j\in J}\zeta_{j}|g_{j}|^{2}\|_{L^{\infty}(\mathbb{T}^{n})}\leqslant C\lambda^{n-1}\epsilon(\lambda).
		\end{equation*}
	\end{theorem}
	
	By Frank-Sabin's duality principle \cite[Lemma 3]{FS2}, Theorem \ref{Thm2} is equivalent to an estimate involving the Schatten norm for compact operators, and it is this equivalent form that we will actually prove. Before giving that statement, let us first recall some basic facts about the Schatten norm.
	
	Suppose $\mathscr{H}_{1}$ and $\mathscr{H}_{2}$ are two complex Hilbert spaces and let $T: \ \mathscr{H}_{1}\longrightarrow\mathscr{H}_{2}$ be a compact operator. The nonzero eigenvalues of $\sqrt{T^{\ast}T}$ are called the singular values of $T$. They form an at most countable set which we label as $\{\mu_{n}(T)\}_{n\in\mathbb{N}}$. Then for any $\alpha>0$, the Schatten-$\alpha$ space $\mathfrak{S}^{\alpha}(\mathscr{H}_{1}, \mathscr{H}_{2})$ is defined as the set of all compact operators $T: \ \mathscr{H}_{1}\longrightarrow\mathscr{H}_{2}$ such that $\sum_{n\in\mathbb{N}}(\mu_{n}(T))^{\alpha}<\infty$, that is, such that $\mathrm{Tr}(T^{\ast}T)^{\alpha/2}<\infty$. For any $T\in\mathfrak{S}^{\alpha}(\mathscr{H}_{1}, \mathscr{H}_{2})$, we call
	\[\|T\|_{\mathfrak{S}^{\alpha}(\mathscr{H}_{1}, \mathscr{H}_{2})}\ :=\left(\sum_{n\in\mathbb{N}}(\mu_{n}(T))^{\alpha}\right)^{1/\alpha}=(\mathrm{Tr}(T^{\ast}T)^{\alpha/2})^{1/\alpha}\]
	the Schatten-$\alpha$ norm of $T$. When $\alpha\geqslant 1$, $\mathfrak{S}^{\alpha}(\mathscr{H}_{1}, \mathscr{H}_{2})$ becomes a Banach space under this norm. If the domain and the target space of the operators are the same complex Hilbert space $\mathscr{H}$, then we will denote the Schatten-$\alpha$ space as $\mathfrak{S}^{\alpha}(\mathscr{H})$ for short. We refer the reader to Simon \cite{Simon} for more properties and applications of the Schatten norm.
	
	By the duality principle just cited, Theorem \ref{Thm2} is equivalent to the statement below.
	
	\begin{theorem} \label{Thm3}
		With the definition of $\chi_{\lambda}^{\epsilon}$ the same as in Theorem \ref{Thm2}, there exists a $C>0$ such that for any $h\in L^{2}(\mathbb{T}^{n})$, \begin{equation} \label{SchattenEquiv}
			\|h\chi_{\lambda}^{\epsilon}\overline{h}\|_{\mathfrak{S}^{1}(L^{2}(\mathbb{T}^{n}))}\leqslant C\lambda^{n-1}\epsilon(\lambda)\|h\|_{L^{2}(\mathbb{T}^{n})}^{2}.
		\end{equation}
	\end{theorem}

	\begin{proof} [Proof of Theorem \ref{Thm3}]
	 Fix any small positive number $\lambda^{-1}<\epsilon<1$. Abusing notation a bit, we denote the spectral projection operator for $-\Delta_{\mathbb{T}^{n}}$ onto the interval $[\lambda^{2}, (\lambda+\epsilon)^{2})$ as $\chi_{\lambda}^{\epsilon}$ too. We are going to prove \begin{equation} \label{goal}
		\|h\chi_{\lambda}^{\epsilon}\overline{h}\|_{\mathfrak{S}^{1}(L^{2}(\mathbb{T}^{n}))}\leqslant C(\epsilon\lambda^{n-1}+(\lambda/\epsilon)^{\frac{n-1}{2}})\|h\|_{L^{2}(\mathbb{T}^{n})}^{2}.
	\end{equation}
	Substituting $\epsilon=\lambda^{-\frac{n-1}{n+1}}$ which leads to the optimal result, we get Theorem \ref{Thm3}.
	
	In order to prove \eqref{goal}, pick an even nonnegative function $a\in\mathscr{S}(\mathbb{R})$ which satisfies $a(0)=1$ and whose Fourier transform is supported in $(-1, 1)$. Indeed, if $\gamma$ is a nonzero even function in $C_{0}^{\infty}((-\frac{1}{2}, \frac{1}{2}))$, then a constant multiple of the inverse Fourier transform of $\gamma\ast\gamma$ has the above required properties. Let $b=a^{2}$. Then $b$ is also an even nonnegative Schwartz class function that satisfies $b(0)=1$ but whose Fourier transform is supported in $(-2, 2)$. Consider the operator $b(\epsilon^{-1}(\sqrt{-\Delta_{\mathbb{T}^{n}}}-\lambda))$
	defined by the spectral theorem. We claim that it suffices to show \eqref{goal} with $\chi_{\lambda}^{\epsilon}$ replaced by $b(\epsilon^{-1}(\sqrt{-\Delta_{\mathbb{T}^{n}}}-\lambda))$, that is, we content ourselves by proving
	\begin{equation} \label{goal2}
		\|hb(\epsilon^{-1}(\sqrt{-\Delta_{\mathbb{T}^{n}}}-\lambda))\overline{h}\|_{\mathfrak{S}^{1}(L^{2}(\mathbb{T}^{n}))}\leqslant C(\epsilon\lambda^{n-1}+(\lambda/\epsilon)^{\frac{n-1}{2}})\|h\|_{L^{2}(\mathbb{T}^{n})}^{2}.
	\end{equation}
	Indeed, since $b(0)=1$, there exists a $\delta>0$ such that $b(t)\geqslant\frac{1}{2}$ on $[0, \delta]$. Thus we obtain the operator inequality \begin{equation*} 
		0\leqslant h\chi_{\lambda}^{\delta\epsilon}\overline{h}\leqslant 2hb(\epsilon^{-1}(\sqrt{-\Delta_{\mathbb{T}^{n}}}-\lambda))\overline{h},
	\end{equation*}
	where as just pointed out, $\chi_{\lambda}^{\delta\epsilon}$ means the spectral projection operator onto the band $[\lambda^{2}, (\lambda+\delta\epsilon)^{2})$. Hence \[\|h\chi_{\lambda}^{\delta\epsilon}\overline{h}\|_{\mathfrak{S}^{1}(L^{2}(\mathbb{T}^{n}))}\leqslant 2\|hb(\epsilon^{-1}(\sqrt{-\Delta_{\mathbb{T}^{n}}}-\lambda))\overline{h}\|_{\mathfrak{S}^{1}(L^{2}(\mathbb{T}^{n}))}.\]
	Summing up $\delta^{-1}$ such estimates as \eqref{goal2} gives the desired estimate \eqref{goal}.
	
	\vspace{2ex}
	
	Now we embark on the proof of \eqref{goal2}. Denote the kernel of $a(\epsilon^{-1}(\sqrt{-\Delta_{\mathbb{T}^{n}}}-\lambda))$ by $A(x, y)$. Utilizing the trace class bound \begin{equation} \label{infinity} \begin{aligned}
			\|hb(\epsilon^{-1}(\sqrt{-\Delta_{\mathbb{T}^{n}}}-\lambda))\overline{h}\|_{\mathfrak{S}^{1}(L^{2}(\mathbb{T}^{n}))}&=\|ha(\epsilon^{-1}(\sqrt{-\Delta_{\mathbb{T}^{n}}}-\lambda))\|^{2}_{\mathfrak{S}^{2}(L^{2}(\mathbb{T}^{n}))}\\
			&=\int_{\mathbb{T}^{n}}\int_{\mathbb{T}^{n}}|h(x)|^{2}|A(x, y)|^{2}dxdy,
		\end{aligned}
	\end{equation}
	we see that it is enough to bound $\|A(x, y)\|_{L^{\infty}L^{2}(\mathbb{T}^{n}\times\mathbb{T}^{n})}$, as the right side of the above estimate \eqref{infinity} is majorized by $\|A(x, y)\|_{L^{\infty}L^{2}(\mathbb{T}^{n}\times \mathbb{T}^{n})}^{2}\|h\|_{L^{2}(\mathbb{T}^{n})}^{2}$. Furthermore, by the spectral theorem and orthogonality, \begin{equation*}
		\int_{\mathbb{T}^{n}}|A(x, y)|^{2}dy\leqslant C|A(x, x)|, \ \forall x\in\mathbb{T}^{n},
	\end{equation*}
	hence we are left to bound $|A(x, x)|$, all values of $A(x, y)$ on the diagonal. We basically apply the strategy in \cite{BSSY}, see also \cite[Section 3.5]{SoggeHangzhou}. Succinctly speaking, it will be done by dividing the operator $a(\epsilon^{-1}(\sqrt{-\Delta_{\mathbb{T}^{n}}}-\lambda))$ into several parts and estimate them piece by piece.
	
	Write \begin{equation*}
		a(\epsilon^{-1}(\sqrt{-\Delta_{\mathbb{T}^{n}}}-\lambda))=\frac{\epsilon}{\pi}\int_{-\infty}^{\infty}\hat{a}(\epsilon t)e^{-it\lambda}\mathrm{cos}(t\sqrt{-\Delta_{\mathbb{T}^{n}}})dt-a(\epsilon^{-1}(-\sqrt{-\Delta_{\mathbb{T}^{n}}}-\lambda)).
	\end{equation*}
	Note that $|a(\epsilon^{-1}(-\tau-\lambda))|\leqslant C_{N}(1+\epsilon^{-1}|\tau+\lambda|)^{-N}$ for any $N\in\mathbb{N}$, since $a$ is a Schwartz class function. Hence expectedly, the $a(\epsilon^{-1}(-\sqrt{-\Delta_{\mathbb{T}^{n}}}-\lambda))$ part will play a trivial role in what follows.
	
	For the principal part $\frac{\epsilon}{\pi}\int_{-\infty}^{\infty}\hat{a}(\epsilon t)e^{-it\lambda}\mathrm{cos}(t\sqrt{-\Delta_{\mathbb{T}^{n}}})dt$, we choose an even function $c\in C_{0}^{\infty}(\mathbb{R})$ with the properties 
	\[c(t)=1, \quad |t|\leqslant 1; \quad \text{and} \quad c(t)=0, \quad |t|\geqslant 2,\]
	and use this $c$ to split this part further as $I_{1}+I_{2}$, where \begin{equation*}
		I_{1}=\frac{\epsilon}{\pi}\int_{-\infty}^{\infty}c(t)\hat{a}(\epsilon t)e^{-it\lambda}\mathrm{cos}(t\sqrt{-\Delta_{\mathbb{T}^{n}}})dt,
	\end{equation*}
	and \begin{equation*}
		I_{2}=\frac{\epsilon}{\pi}\int_{-\infty}^{\infty}(1-c(t))\hat{a}(\epsilon t)e^{-it\lambda}\mathrm{cos}(t\sqrt{-\Delta_{\mathbb{T}^{n}}})dt.
	\end{equation*}
	
	Denote the kernels of the operators $I_{1}$ and $I_{2}$ by $I_{1}(x, y)$ and $I_{2}(x, y)$ respectively, and that of the negligible part $a(\epsilon^{-1}(-\sqrt{-\Delta_{\mathbb{T}^{n}}}-\lambda))$ by $J(x, y)$. First, recall the well-known estimate $\|e_{j}\|_{L^{\infty}(M)}\leqslant(1+\lambda_{j})^{\frac{n-1}{2}},\ \forall j\geqslant 0$ for eigenfunctions of the Laplacian on any $n$ dimensional compact manifold $M$, and the Weyl formula for any $n$ dimensional compact manifold which states that the number of eigenvalues $\leqslant l$ of $-\Delta_{g}$ is $Cl^{n}+O(l^{n-1})$, $\forall l\geqslant 1$ (see for instance, \cite[Section 3.2]{SoggeHangzhou}), then \begin{equation} \label{J} \begin{aligned}
			J(x, x)&=\sum_{j=0}^{\infty}a(\epsilon^{-1}(-\lambda_{j}-\lambda))|e_{j}(x)|^{2}\\
			&\leqslant C_{N}\sum_{l=0}^{\infty}(1+\epsilon^{-1}|l+\lambda|)^{-N}l^{2n-2}\\
			&\leqslant C_{N}(1+\lambda)^{-N+2n}
		\end{aligned}
	\end{equation}
	for any $N\in\mathbb{N}$.
	
	So we focus on $I_{1}$ and $I_{2}$. Identifying $\mathbb{T}^{n}$ with $D=(-\frac{1}{2}, \frac{1}{2}]^{n}$, we are going to apply the following variant of the Poisson summation formula: \begin{equation} \label{formula}
		\mathrm{cos}(t\sqrt{-\Delta_{\mathbb{T}^{n}}})(x, y)=\sum_{k\in\mathbb{Z}^{n}}\mathrm{cos}(t\sqrt{-\Delta_{\mathbb{R}^{n}}})(x-y+k), \quad\forall x, y\in D,
	\end{equation}
	which can be found in \cite{Hlawka}. See also \cite[Section 3.5]{SoggeHangzhou}.
	
	To estimate $I_{1}(x, x)$, by the Huygens Principle, $\mathrm{cos}(t\sqrt{-\Delta_{\mathbb{R}^{n}}})(x, y)$ is supported where $|x-y|\leqslant|t|$. Therefore, by the support property of $c(t)$, $I_{1}(x, x)$ is equal to \begin{equation}\label{I21}
		\frac{2\epsilon}{(2\pi)^{n+1}}\sum_{|k|\leqslant 2}\int_{\mathbb{R}^{n}}\int_{-\infty}^{\infty}e^{ik\cdot\xi}c(t)\hat{a}(\epsilon t)e^{-it\lambda}\mathrm{cos}(t|\xi|)dtd\xi.
	\end{equation}
	Denote by $H_{\epsilon}(\tau)$ the Fourier transform of $c(t)\hat{a}(\epsilon t)$, which is easily seen to lie in a bounded set of the space $\mathscr{S}(\mathbb{R})$ for all $\epsilon<1$. When $k=(0, \cdots, 0)$, the corresponding summand in the above \eqref{I21} then equals \begin{equation*}
		\frac{\epsilon}{(2\pi)^{n+1}}\int_{\mathbb{R}^{n}}[H_{\epsilon}(\lambda-|\xi|)+H_{\epsilon}(\lambda+|\xi|)]d\xi,
	\end{equation*}
	and can be controlled by $\epsilon\lambda^{n-1}$ after a use of polar coordinates. When $1\leqslant|k|\leqslant 2$, the corresponding summands are \begin{equation} \label{I212}
		\frac{\epsilon}{(2\pi)^{n+1}}\sum_{1\leqslant|k|\leqslant 2}\int_{\mathbb{R}^{n}}e^{ik\cdot\xi}[H_{\epsilon}(\lambda-|\xi|)+H_{\epsilon}(\lambda+|\xi|)]d\xi.
	\end{equation}
	
	For these summands, we use the following expression for the Fourier transform of surface-carried measures on the standard round sphere in $\mathbb{R}^{n}$ (it is in fact true for any smooth hypersurface in $\mathbb{R}^{n}$ with nonvanishing Gaussian curvature, cf. \cite[Section 1.2]{FIO}): Let $d\mu$ be a $C_{0}^{\infty}$ measure on the sphere (for instance, the induced Lebesgue measure), then \begin{equation} \label{surface}
		|\widehat{d\mu}(\xi)|=\sum_{\pm}(1+|\xi|)^{-(n-1)/2}m_{\pm}(|\xi|)e^{\pm i|\xi|},
	\end{equation}
	where for $r\geqslant 1$, $m_{\pm}(r)$ both satisfy 
	\[\left |\frac{d^{j}}{dr^{j}}m_{\pm}(r)\right |\leqslant C_{j}r^{-j}, \quad j=0, 1, 2, \dots.\]
	With this expression, by using polar coordinates, \eqref{I212} equals \begin{equation*}
		\frac{\epsilon}{(2\pi)^{n+1}}\sum_{1\leqslant|k|\leqslant 2}\int_{\mathbb{R}^{n}}e^{ir|k|}r^{\frac{n-1}{2}}|k|^{-\frac{n-1}{2}}[H_{\epsilon}(\lambda-r)+H_{\epsilon}(\lambda+r)]m_{\pm}(r|k|)dr,
	\end{equation*}
	and is easily seen to be bounded by $C\epsilon\lambda^{\frac{n-1}{2}}$. Therefore, \begin{equation} \label{I1}
		|I_{1}(x, x)|\leqslant C\epsilon\lambda^{n-1}.
	\end{equation}
	
	Now we come to analyzing $I_{2}(x, x)$. By the Huygens Principle and the support properties of $(1-c(t))$ and $\hat{a}(\epsilon t)$, we may write $I_{2}(x, x)$ as \begin{equation}
		I_{2}(x, x)=\frac{2\epsilon}{(2\pi)^{n+1}}\sum_{1\leqslant |k|\leqslant 1/\epsilon}\int_{\mathbb{R}^{n}}\int_{-\infty}^{\infty}e^{ik\cdot\xi}(1-c(t))\hat{a}(\epsilon t)e^{-it\lambda}\mathrm{cos}(t|\xi|)dtd\xi.
	\end{equation}
	In order to estimate the right side of this equality, we further split it as the difference of \begin{equation*}
		I_{21}(x, x)=\frac{2\epsilon}{(2\pi)^{n+1}}\sum_{1\leqslant |k|\leqslant 1/\epsilon}\int_{\mathbb{R}^{n}}\int_{-\infty}^{\infty}e^{ik\cdot\xi}c(t)\hat{a}(\epsilon t)e^{-it\lambda}\mathrm{cos}(t|\xi|)dtd\xi
	\end{equation*}
	and \begin{equation} \label{I22}
		I_{22}(x,x)=\frac{2\epsilon}{(2\pi)^{n+1}}\sum_{1\leqslant |k|\leqslant 1/\epsilon}\int_{\mathbb{R}^{n}}\int_{-\infty}^{\infty}e^{ik\cdot\xi}\hat{a}(\epsilon t)e^{-it\lambda}\mathrm{cos}(t|\xi|)dtd\xi.
	\end{equation}
	
	$I_{21}(x, x)$ may be treated in the same way as we did for $I_{1}(x, x)$, as $c(t)$ vanishes when $|t|\geqslant 2$ and hence only summands with $1\leqslant|k|\leqslant 2$ are nonzero. Thus, $|I_{21}(x, x)|\leqslant C\epsilon\lambda^{\frac{n-1}{2}}$. So we turn to $I_{22}(x, x)$.
	
	$I_{22}(x, x)$ is equal to \begin{equation*}
		\frac{1}{(2\pi)^{n+1}}\sum_{1\leqslant |k|\leqslant 1/\epsilon}\int_{\mathbb{R}^{n}}e^{ik\cdot\xi}[a(\epsilon^{-1}(-\lambda+|\xi|))+a(\epsilon^{-1}(-\lambda-|\xi|))]d\xi.
	\end{equation*}
	Using \eqref{surface} again and that $a$ is a Schwartz class function, the above sum can be bounded by \begin{equation*}
		\begin{aligned}
			&\sum_{1\leqslant |k|\leqslant 1/\epsilon}\int_{0}^{\infty}r^{\frac{n-1}{2}}|k|^{-\frac{n-1}{2}}[(1+\epsilon^{-1}|\lambda-r|)^{-n}+(1+\epsilon^{-1}|\lambda+r|)^{-n}]dr\\
			&\leqslant C\sum_{1\leqslant |k|\leqslant1/\epsilon} \epsilon\lambda^{\frac{n-1}{2}}|k|^{-\frac{n-1}{2}}\\
			&\leqslant C(\lambda/\epsilon)^{\frac{n-1}{2}}
		\end{aligned}
	\end{equation*}
	Therefore, $|I_{22}(x, x)|\leqslant C(\lambda/\epsilon)^{\frac{n-1}{2}}$, and hence \begin{equation} \label{I2}
		|I_{2}(x, x)|\leqslant(\lambda/\epsilon)^{\frac{n-1}{2}}.
	\end{equation}
	
	Combining the estimates \eqref{I1}, \eqref{I2} and \eqref{J}for $I_{1}(x, x)$, $I_{2}(x, x)$ and $J(x,x)$, we obtain the desired estimate \eqref{goal2}, thereby finishing the proof of Theorem \ref{Thm3}. We remark here that by using a different refined estimate for the nonlocal term $I_{2}$ via multidimensional Weyl sum estimate, we can improve slightly the power of $\lambda$ in the expression for $\epsilon(\lambda)$ in our main theorem, which we do not include here. See \cite{Hickman} for details in the case of single functions in a slightly different approach.
\end{proof}

\vspace{1ex}
	
	\begin{proof}[Proof of Corollary \ref{Cor1}] Recall our notation: for any small positive number $\lambda^{-1}<\epsilon<1$, $\chi_{\lambda}^{\epsilon}$ means the spectral projection operator for $-\Delta_{\mathbb{T}^{n}}$ onto the interval $[\lambda^{2}, (\lambda+\epsilon)^{2})$. We claim that there exists a $C>0$, not depending on $\epsilon$, such that for any $h\in L^{n+1}(\mathbb{T}^{n})$, we have \begin{equation} \label{critical1}
		\|h\chi_{\lambda}^{\epsilon}\overline{h}\|_{\mathfrak{S}^{n+1}(L^{2}(\mathbb{T}^{n}))}\leqslant C\lambda^{\frac{n-1}{n+1}}\|h\|_{L^{n+1}(\mathbb{T}^{n})}^{2}.
	\end{equation}
	Then, interpolating with \eqref{SchattenEquiv} in Theorem \ref{Thm3} yields the following estimate at any $\frac{2(n+1)}{n-1}\leqslant p\leqslant\infty$: \begin{equation*}
		\|h\chi_{\lambda}^{\epsilon}\overline{h}\|_{\mathfrak{S}^{\alpha(p)'}(L^{2}(\mathbb{T}^{n}))}\leqslant C\lambda^{2\sigma(p)}\epsilon(\lambda)^{\left(1-\frac{2(n+1)}{p(n-1)}\right)}\|h\|_{L^{2p/(p-2)}(\mathbb{T}^{n})}^{2}, \ \forall h\in L^{2p/(p-2)}(\mathbb{T}^{n}),
	\end{equation*}
where $\alpha(p)=\frac{p(n-1)}{2n}$ and $\sigma(p)=n(\frac{1}{2}-\frac{1}{p})-\frac{1}{2}$. This together with the duality principle implies the Proposition below, from which Corollary \ref{Cor1} clearly follows.

\begin{proposition}
	Suppose the operator $\chi_{\lambda}^{\epsilon}$ and the space $U_{\lambda}^{\epsilon}$ are as in Theorem \ref{Thm1}. Let $\frac{2(n+1)}{n-1}\leqslant p\leqslant\infty$. Then there exists a $C>0$ such that for any orthonormal system $\{g_{j}(x)\}_{j\in J}$ of $U_{\lambda}^{\epsilon}$ and any sequence of complex numbers $\{\zeta_{j}\}_{j\in J}$, we have \begin{equation*}
		\|\sum_{j\in J}\zeta_{j}|g_{j}|^{2}\|_{L^{p/2}(\mathbb{T}^{n})}\leqslant C\lambda^{2\sigma(p)}\epsilon(\lambda)^{\left(1-\frac{2(n+1)}{p(n-1)}\right)}\left(\sum_{j\in J}|\zeta_{j}|^{\alpha(p)}\right)^{1/\alpha(p)},
	\end{equation*}
	where $\alpha(p)=\frac{p(n-1)}{2n}$ and $\sigma(p)=n(\frac{1}{2}-\frac{1}{p})-\frac{1}{2}$.
\end{proposition}
	
It therefore remains to prove the claim. In \cite{FS}, Frank-Sabin showed that for a general compact manifold $(M, g)$, \begin{equation} \label{FS2}
	\|h\chi_{\lambda}\overline{h}\|_{\mathfrak{S}^{n+1}(L^{2}(M))}\leqslant C\lambda^{\frac{n-1}{n+1}}\|h\|_{L^{n+1}(M)}^{2},
\end{equation} 
where $\chi_{\lambda}$ means the spectral projection operator for $-\Delta_{\mathbb{T}^{n}}$ onto the band $[\lambda^{2}, (\lambda+1)^{2})$. We will appeal to this estimate to prove our claim. 

Observe that by a similar line of reasoning as in the $p=\infty$ case, it suffices to prove \begin{equation} \label{critical}
	\|hb(\epsilon^{-1}(\sqrt{-\Delta_{\mathbb{T}^{n}}}-\lambda))\overline{h}\|_{\mathfrak{S}^{n+1}(L^{2}(\mathbb{T}^{n}))}\leqslant C\lambda^{\frac{n-1}{n+1}}\|h\|_{L^{n+1}(\mathbb{T}^{n})}^{2},
\end{equation}
where $b$ is the auxiliary function used there. Since $b$ is a Schwartz class function, \begin{equation*}
	b(\epsilon^{-1}(\tau-\lambda))\leqslant C_{N}(1+\epsilon^{-1}|\tau-\lambda|)^{-N}.
\end{equation*}
Thus, we get the operator inequality \begin{equation} \label{Opine} \begin{aligned}
		0\leqslant hb(\epsilon^{-1}(\sqrt{-\Delta_{\mathbb{T}^{n}}}-\lambda))\overline{h}&\leqslant C_{N}h[(1+\epsilon^{-1}|\sqrt{-\Delta_{\mathbb{T}^{n}}}-\lambda|)^{-2}]\overline{h}\\
		&\leqslant C_{N}(1+\epsilon^{-1}|l-\lambda|)^{-2}\sum_{l=0}^{\infty}h\chi_{l}\overline{h}.
	\end{aligned}
\end{equation}
Applying Frank-Sabin's work \eqref{FS2}, we derive \begin{equation*} 
	\begin{aligned}
		&\|hb(\epsilon^{-1}(\sqrt{-\Delta_{\mathbb{T}^{n}}}-\lambda))\overline{h}\|_{\mathfrak{S}^{n+1}(L^{2}(\mathbb{T}^{n}))}\leqslant C\sum_{l=0}^{\infty}(1+\epsilon^{-1}|l-\lambda|)^{-2}\|h\chi_{l}\overline{h}\|_{\mathfrak{S}^{n+1}(L^{2}(\mathbb{T}^{n}))}\\
		&\leqslant C\sum_{l=0}^{\infty}l^{\frac{n-1}{n+1}}(1+\epsilon^{-1}|l-\lambda|)^{-2}\|h\|_{L^{n+1}(\mathbb{T}^{n})}^{2}\\ 
		&\leqslant C\lambda^{\frac{n-1}{n+1}}\|h\|_{L^{n+1}(\mathbb{T}^{n})}^{2},
	\end{aligned}
\end{equation*}
where the constant $C$ does not rely on $\epsilon$, as desired.
\end{proof}
	
	\maketitle
	\section{Proof of Theorem \ref{Thm6}}
	
	As in the case of the flat torus, Theorem \ref{Thm6} is an easy consequence of the following stronger result:
	
	\begin{theorem} \label{Thm7}
		Let $\chi_{\lambda}^{\epsilon}$ be the spectral projection operator for $-\Delta_{g}$ onto the interval $[\lambda^{2}, (\lambda+\epsilon(\lambda))^{2})$, where $\epsilon(\lambda)=K\frac{1}{\mathrm{ln}\lambda}$ with $K$ being some fixed constant. Let $U_{\lambda}^{\epsilon}=\chi_{\lambda}^{\epsilon}(L^{2}(M))$.  Then, there exists a $C>0$ such that for any orthonormal system $\{g_{j}(x)\}_{j\in J}$ of $U_{\lambda}^{\epsilon}$ and any sequence of complex numbers $\{\zeta_{j}\}_{j\in J}$, we have \begin{equation*}
			\|\sum_{j\in J}\zeta_{j}|g_{j}|^{2}\|_{L^{\infty}(M)}\leqslant C\lambda^{n-1}\epsilon(\lambda).
		\end{equation*}
	\end{theorem}
	
	And again, by Frank-Sabin's duality principle \cite[Lemma 3]{FS2}, the above theorem is equivalent to a statement involving the Schatten norm, and it is this equivalent version that we are actually going to prove.
	
	\begin{theorem} \label{Thm8}
		With the definition of $\chi_{\lambda}^{\epsilon}$ the same as in Theorem \ref{Thm7}, there exists a $C>0$ such that for any $h\in L^{2}(M)$, \begin{equation} \label{nonpmain}
			\|h\chi_{\lambda}^{\epsilon}\overline{h}\|_{\mathfrak{S}^{1}(L^{2}(M))}\leqslant C\lambda^{n-1}\epsilon(\lambda)\|h\|_{L^{2}(M)}^{2}.
		\end{equation}
	\end{theorem}

     \begin{proof} [Proof of Theorem \ref{Thm8}]	
	 Abusing notation as before, for any fixed small positive number $\lambda^{-1}<\epsilon<1$, let $\chi_{\lambda}^{\epsilon}$ denote the spectral projection operator for $-\Delta_{g}$ onto the interval $[\lambda^{2}, (\lambda+\epsilon)^{2})$. Theorem \ref{Thm8} follows once we have proved the estimate
	\begin{equation} \label{nonpmain1}
		\|h\chi_{\lambda}^{\epsilon}\overline{h}\|_{\mathfrak{S}^{1}(L^{2}(M))}\leqslant C_{1}(\epsilon\lambda^{n-1}+\mathrm{exp}(C_{2}/\epsilon)\mathrm{max}\{\lambda^{\frac{n-1}{2}}, \lambda^{n-3}\})\|h\|_{L^{2}(M)}^{2}.
	\end{equation}
	Indeed, taking $\epsilon=\frac{1}{\epsilon_{1}\mathrm{ln}\lambda}$ in \eqref{nonpmain1}, where $\epsilon_{1}>0$ is small enough to guarantee
	\[\lambda^{C_{2}\epsilon_{1}+\mathrm{max}\{\frac{n-1}{2}, n-3\}}\leqslant\epsilon\lambda^{n-1}, \quad\forall\lambda\geqslant 1,\]
	we obtain
	\begin{equation} 
		\|h\chi_{\lambda}^{\epsilon}\overline{h}\|_{\mathfrak{S}^{1}(L^{2}(M))}\leqslant C\epsilon\lambda^{n-1}\|h\|_{L^{2p/(p-2)}(M)}^{2},
	\end{equation}
	with $\epsilon$ being a fixed multiple of $\frac{1}{\mathrm{ln}\lambda}$, which is what we want.
	
	It remains to prove \eqref{nonpmain1}. We recall the auxiliary functions in Section 2. $a\in\mathscr{S}(\mathbb{R})$ is an even nonnegative function satisfying $a(0)=1$ and $\mathrm{supp}\ \hat{a}\subseteq(-1, 1)$, and $b=a^{2}$. Arguing in the same way, we deduce that in order to prove \eqref{nonpmain1}, it suffices to show \begin{multline} \label{a3}
		\|hb(\epsilon^{-1}(\sqrt{-\Delta_{g}}-\lambda))\overline{h}\|_{\mathfrak{S}^{1}(L^{2}(M))}\leqslant\\ C_{1}(\epsilon\lambda^{n-1}+\mathrm{exp}(C_{2}/\epsilon)\mathrm{max}\{\lambda^{\frac{n-1}{2}}, \lambda^{n-3}\})\|h\|_{L^{2}(M)}^{2}.
	\end{multline}
	
	For this, following the lines of the proof of the flat torus case, we see that it is enough to bound $A(x, x), \forall x\in M$ by $C_{1}(\epsilon\lambda^{n-1}+\mathrm{exp}(C_{2}/\epsilon)\mathrm{max}\{\lambda^{\frac{n-1}{2}}, \lambda^{n-3}\})$, where $A(x, y)$ is the kernel of the operator $a(\epsilon^{-1}(\sqrt{-\Delta_{g}}-\lambda))$. Writing \begin{equation*}
		a(\epsilon^{-1}(\sqrt{-\Delta_{g}}-\lambda))=\frac{\epsilon}{\pi}\int_{-\infty}^{\infty}\hat{a}(\epsilon t)e^{-it\lambda}\mathrm{cos}(t\sqrt{-\Delta_{g}})dt-a(\epsilon^{-1}(-\sqrt{-\Delta_{g}}-\lambda))
	\end{equation*} and recalling that the kernel of $a(\epsilon^{-1}(-\sqrt{-\Delta_{g}}-\lambda))$  is $O(\lambda^{-N})$ for any $N\in\mathbb{N}$, we are left to bound the kernel of
	\[\frac{\epsilon}{\pi}\int_{-\infty}^{\infty}\hat{a}(\epsilon t)e^{-it\lambda}\mathrm{cos}(t\sqrt{-\Delta_{g}})dt,\]
	on the diagonal.
	
	In order to accomplish this, similar to the flat torus case, we use a Poisson-type formula that relates the kernel of the wave operator $\mathrm{cos}(t\sqrt{-\Delta_{g}})$ to a periodic sum of the kernels of the wave operator on the universal cover of $(M, g)$. By Hadamard's theorem (cf. \cite[Section 7.3]{DoCarmo}), since $(M, g)$ has nonpositive sectional curvature, the universal cover of $(M, g)$ is $\mathbb{R}^{n}$. Let $p:\ \mathbb{R}^{n}\rightarrow M$ be the covering map, and denote the pullback of $g$ via $p$ by $\tilde{g}$. For notational convenience, we use $(\tilde{M}, \tilde{g})$ to represent the universal cover $\mathbb{R}^{n}$ endowed with metric $\tilde{g}$, and letters with a tilde are to denote objects associated with $\tilde{M}$. Let $D\subseteq\tilde{M}$ be a fundamental domain, which can be identified with $M$ via the covering map $p$. Recall also that by a deck transformation, we mean a homeomorphism $\gamma:\ \tilde{M}\rightarrow\tilde{M}$ satisfying $p=p\circ\gamma$. All deck transformations form a group under composition, which we denote as $\Gamma$. We then have $M\simeq\tilde{M}/\Gamma$. The formula we need is \begin{equation}
		\mathrm{cos}(t\sqrt{-\Delta_{g}})(\tilde{x}, \tilde{y})=\sum_{\gamma\in\Gamma}\mathrm{cos}(t\sqrt{-\Delta_{\tilde{g}}})(\tilde{x}, \gamma(\tilde{y})), \quad\forall \tilde{x}, \tilde{y}\in D.
	\end{equation}
	We refer the reader to \cite[Section 3.6]{SoggeHangzhou} for more background information and the derivation of this formula.
	
	Therefore, it is enough to prove 
	\begin{multline} \label{sum}
			\sum_{\gamma\in\Gamma}\left |\int_{-\infty}^{\infty}\epsilon\hat{a}(\epsilon t)e^{-it\lambda}\mathrm{cos}(t\sqrt{-\Delta_{g}})(\tilde{x}, \gamma(\tilde{x}))dt\right |\\
			\leqslant C_{1}(\epsilon\lambda^{n-1}+\mathrm{exp}(C_{2}/\epsilon)\mathrm{max}\{\lambda^{\frac{n-1}{2}}, \lambda^{n-3}\}), \quad\forall\tilde{x}\in D.
	\end{multline}
	By the Huygens principle and the support property of $\hat{a}$, a summand in \eqref{sum} is nonzero only if $d_{g}(\tilde{x}, \gamma(\tilde{x}))\leqslant \epsilon^{-1}$. Since the volume of any geodesic ball of radius $r$ in $\tilde{M}$ is $O(\mathrm{exp}(C'r))$ for some uniform constant $C'$, and $d_{\tilde{g}}(\gamma_{1}(\tilde{x}), \gamma_{2}(\tilde{x}))\geqslant c$, $\forall\gamma_{1}\neq\gamma_{2}$,  $\gamma_{1}, \gamma_{2}\in\Gamma$, $\forall\tilde{x}\in\tilde{M}$, we deduce that there are $O(\mathrm{exp}(C''/\epsilon))$ nonzero terms in the above sum for some constant $C''$. Thus, \eqref{sum} would be proved if we could show
	\begin{equation} \label{sum1}
		\left |\int_{-\infty}^{\infty}\epsilon\hat{a}(\epsilon t)e^{-it\lambda}\mathrm{cos}(t\sqrt{-\Delta_{g}})(\tilde{x}, \tilde{x})dt\right |\leqslant C_{1}(\epsilon\lambda^{n-1}+\mathrm{exp}(C_{2}/\epsilon)\lambda^{n-3})
	\end{equation}
	and \begin{multline} \label{sum2}
		\left |\int_{-\infty}^{\infty}\epsilon\hat{a}(\epsilon t)e^{-it\lambda}\mathrm{cos}(t\sqrt{-\Delta_{g}})(\tilde{x}, \gamma(\tilde{x}))dt\right |\leqslant C_{1}\mathrm{exp}(C_{2}/\epsilon)\lambda^{\frac{n-1}{2}}, \\\forall\gamma\neq\mathrm{Identity}.
	\end{multline}
	Because $(\tilde{M}, \tilde{g})$ is a Riemannian manifold without conjugate points, these two inequalities can be proved using the Hadamard parametrix for $\partial_{t}^{2}-\Delta_{g}$. In fact, \eqref{sum1} and \eqref{sum2} are just (3.6.8) and (3.6.9) respectively in \cite{SoggeHangzhou}[Section 3.6]. This concludes our proof of Theorem \ref{Thm8}.
	
\end{proof}

\vspace{1ex}
	
	\begin{proof}[Proof of Corollary \ref{Cor2}] Since our proof of \eqref{critical1} does not rely on the specific type of compact manifold we worked on, it is true for any compact manifold. With this in mind, Corollary \ref{Cor2} follows by interpolating \eqref{critical1} with \eqref{nonpmain} and applying the duality principle as we did for the proof of Corollary \ref{Cor1}.
\end{proof}

	\vspace{0.5cm}
	
	\bibliography{OrthonormalSystemsV3}
	
	\bibliographystyle{plain}
	
\end{document}